\documentclass[12pt]{amsart}

\usepackage{amsmath, amssymb, amscd, newlfont}

\usepackage[mathscr]{eucal}

\numberwithin{equation}{section}

\newtheorem{Prop}[equation]{Proposition}
\newtheorem{Lem}[equation]{Lemma}
\newtheorem{Def}[equation]{Definition}
\newtheorem{Thm}[equation] {Theorem}
\newtheorem{Cor}[equation]{Corollary}

\newtheorem{Assumptions}[equation]{Assumptions}

\title
   [Schwartz Space of $\Gamma\backslash G$]
   {Some Results on the Schwartz Space of $\Gamma\backslash G$}
\author{ Goran Mui\' c}

\address{ Department of Mathematics,
University of Zagreb,
Bijeni\v cka 30, 10000 Zagreb,
Croatia}
\email{gmuic@math.hr}

\thanks{The  author acknowledges Croatian Science Foundation grant no. 9364.}

\begin{document}
\dedicatory{to Sibe Marde\v si\' c, in memoriam}

\maketitle
\begin{abstract}
  Let $G$ be  a connected semisimple Lie group with finite center. Let $\Gamma \subset G$ be a discrete subgroup. We study closed admissible
  irreducible subrepresentations of the space of distributions $\mathcal S(\Gamma \backslash G)'$ defined by Casselman \cite{casselman-1}, and their
  relations to automorphic forms. 
\end{abstract}

\section{Introduction}
Let $G$ be  a connected semisimple Lie group with finite center. Let $K$ be the maximal compact subgroup of $G$, and
$\cal Z(\mathfrak g_{\mathbb C})$ the center of the universal enveloping algebra of the complexification of
the Lie algebra $\mathfrak g$ of $G$. Let $\Gamma \subset G$ be a discrete subgroup.
For example, it could be  a trivial group. But the main example is given by the following

\begin{Assumptions}\label{int-paf-00}
We assume that  $G$ is a group of $\mathbb R$--points of a
semisimple algebraic group $\mathcal G$ defined over $\mathbb Q$. Assume that $G$ is not
compact and connected.
Let $\Gamma\subset G$ be congruence subgroup with respect to the arithmetic structure
given by the
fact that $\mathcal G$ defined over $\mathbb Q$ (see \cite{BJ}).
\end{Assumptions}

In \cite{casselman-1}, Casselman has defined the Schwartz space
$\mathcal S(\Gamma\backslash G)$ (see Section \ref{src} for definition). It is obvious that $G$ acts on the right. The corresponding
representation is a smooth representation of moderate growth (\cite{casselman}, \cite{W2}). The main object of the interest is the strong topological dual
space $\mathcal S(\Gamma\backslash G)'$. This is the space of all continuous linear functionals on  $\mathcal S(\Gamma\backslash G)$ equipped with
the strong topology. By general theory of topological vector spaces, the space $\mathcal S\left(\Gamma\backslash G\right)'$ is a  complete
locally convex   vector space.
The natural action of $G$ on
$\mathcal S\left(\Gamma\backslash G\right)'$ is continuous. The usual representation--theoretic arguments are valid
there (\cite{hc}, Section 2).

The main interest in the space $\mathcal S(\Gamma\backslash G)'$ is that its Garding space can be identified with the space of functions of uniform
moderate growth  $\mathcal A_{umg}(\Gamma \backslash G)$ (see (UMG-1) and (UMG-2) in Section \ref{src} for the definition).
Under Assumption \ref{int-paf-00}, $\cal Z(\mathfrak g_{\mathbb C})$--finite $\mathcal A_{umg}(\Gamma \backslash G)$  are smooth automorphic forms on $G$ for
$\Gamma$.  Also, $\cal Z(\mathfrak g_{\mathbb C})$--finite and $K$--finite on the right in $\mathcal A_{umg}(\Gamma \backslash G)$
are equal to the  space usual space $\mathcal A(\Gamma\backslash G)$ of $K$--finite automorphic forms for $\Gamma$ \cite{BJ}.

Now, we describe the content of the paper and main results proved in the paper. In Section \ref{paf}, under Assumption \ref{int-paf-00}, we recall the
notion of smooth and $K$--finite automorphic forms. In Section \ref{src}, we describe the results of Casselman \cite{casselman-1} used in the paper.
In Section \ref{srs} we prove some main results in the paper. This section is strongly motivated by a lecture of Wallach \cite{W3}. Some of the results here
are probably well--known, and we present our way of understanding them. We let $(\pi, \mathcal H)$ be an irreducible admissible representation of $G$
acting on the Hilbert space $\mathcal H$. The space of $\mathcal H^\infty$ vectors in $\mathcal H$ is a representation of moderate growth.
The main results of  Section \ref{srs} gives the description of closed irreducible admissible  subrepresentations of $\mathcal S(\Gamma\backslash G)'$
in terms of
continuous $\Gamma$--invariant functionals on $\mathcal H^\infty$ (see Proposition \ref{srs-0}, Theorem \ref{srs-29}). The proofs use deep results of
Casselman and Wallach (\cite{casselman}, \cite{W2}) on smooth globalization of representations at the critical points.
Examples of subrepresentation can explicitly be constructed using Eisenstein series \cite{langlands},  or be shown to exists using
Poincar\' e series (\cite{MuicMathAnn}, \cite{MuicComp}, \cite{MuicJFA}, \cite{MuicSMOOTH}), or the trace formula (\cite{arthur}, \cite{arthur-1}). In Theorem \ref{srs-30}, we prove that the trivial
representation is the only finite--dimensional subrepresentation of $\mathcal S(\Gamma\backslash G)'$   under Assumption \ref{int-paf-00} and assuming that
$G$ has no compact components. In Section \ref{srs-nas}, we study realization inside $\mathcal S(\Gamma\backslash G)'$ of 
irreducible subrepresentations $\mathcal H$ of $L^2(\Gamma \backslash G)$ (see Theorem \ref{srs-21}). In this case,
$\mathcal H^\infty \subset \mathcal A_{umg}(\Gamma\backslash G)$.  The proof of Theorem \ref{srs-21} contains the proof
of the fact that smooth cuspidal automorphic forms are rapidly decreasing. This is proved using methods of Casselman and Wallach. Different proof is
contained in \cite{ms1}. In Theorem \ref{srs-23}, we relate various topologies on $\mathcal H^\infty$ for an irreducible subspace
$\mathcal H\subset L^2(\Gamma \backslash G)$. For example, we prove that  if the sequence of elements
  in $\mathcal H^\infty$, $(\varphi_n)_{n\ge 1}$, converges to $\varphi\in\mathcal H^\infty$ in the standard topology on $\mathcal H^\infty$, then it 
  converges to $\varphi$ in  usual topology on  $C^\infty(G)$ (see the description before the statement of Theorem \ref{srs-32}. In Section \ref{oii}, we study $\Gamma$--invariants in $\mathcal S'(G)$ and their relation to the space $\left(\mathcal S(G)'\right)^\Gamma$ (see Proposition \ref{srs-12}).
  In Proposition \ref{srs-11} we give the interpretaion of the classical construction of automorphic via Poincar\' e series
  (see for example \cite{MuicSMOOTH}) in terms of $\Gamma$--invariants in $\mathcal S'(G)$.

\section{Preliminaries}\label{paf}

In this section we assume that  $G$ is a connected semisimple Lie group
with finite center, and  recall the notion of the norm on $G$.
It is essential for all what follows. 

\medskip

We fix a minimal parabolic subgroup $P=MAN$ of $G$ in
the usual way (see \cite{W1}, Section 2).
We have the Iwasawa decomposition $G=NAK$.

\medskip

We recall the notion of a  norm on the group following \cite{W1}, 2.A.2.
A norm $|| \ ||$ is a continuous function $G\longrightarrow [1, \infty[$ satisfying
    the following properties:
    \begin{itemize}
    \item[(1)] $||x^{-1}||=||x||$, for all $x\in G$;
    \item[(2)]$||x\cdot y||\le ||x||\cdot || y||$, for all $x, y\in G$;
    \item[(3)] the sets $\left\{x\in G; \ \ ||x||\le r \right\}$ are compact  for all $r\ge 1$;
    \item[(4)] $||k_1\exp{(tX)} k_2||=||\exp{(X)}||^t$, for all $k_1, k_2\in K,
      X\in \mathfrak p, \ \ t\ge 0$.
    \end{itemize}
    Any two norms $||\ ||_i$, $i=1,2$, are equivalent: there exist $C, r>0$ such that
    $||x||_1\le C ||x||^r_2$, for all $x\in G$.

    \medskip
    We recall the following lemma:

    \begin{Lem}\label{paf-0} There exists a real number $d_0>0$ such that
      $\int_G ||g||^{-d} dg<\infty$  for $d\ge d_0$. Since $||g|\ge 1$ for all $g\in G$, the lemma
      follows.
    \end{Lem}
    \begin{proof} The existence of $d_0>0$ such that  $\int_G ||g||^{-d_0} dg<\infty$
is proved in (\cite{W1}, Lemma 2.A.2.4). 
      \end{proof}

\vskip .2in 

In the remainder of this section, we assume the following:

\begin{Assumptions}\label{paf-00}
We assume that  $G$ is a group of $\mathbb R$--points of a
semisimple algebraic group $\mathcal G$ defined over $\mathbb Q$. Assume that $G$ is not
compact and connected.
Let $\Gamma\subset G$ be congruence subgroup with respect to the arithmetic structure
given by the
fact that $\mathcal G$ defined over $\mathbb Q$ (see \cite{BJ}).
\end{Assumptions}
The group satisfying the Assumption \ref{paf-00} is a connected semisimple Lie group
with finite center. Also,  $\Gamma$ is a discrete subgroup
of $G$ and it has a finite covolume.

\medskip
An automorphic form (or a $K$--finite automorphic form; see \cite{cogdell}) for $\Gamma$
is a function $f\in C^\infty(G)$ satisfying the following
three conditions (\cite{W3} or \cite{BJ}):

\begin{itemize}
\item[(A-1)] $f$ is  $\cal Z(\mathfrak g_{\mathbb C})$--finite and $K$--finite on the right;
\item[(A-2)] $f$ is left--invariant under $\Gamma$ i.e., $f(\gamma x)=f(x)$ for all 
  $\gamma\in \Gamma$, $x\in G$;
\item[(A-3)] there exists $r\in\mathbb R$, $r>0$
  such that for each $u\in \mathcal U(\mathfrak g_{\mathbb C})$ there exists a constant
  $C_u>0$ such that $\left|u.f(x)\right|\le C_u \cdot ||x||^r$, for all $x\in G$.
\end{itemize}
A smooth automorphic form (see \cite{casselman-1}, \cite{cogdell}) for $\Gamma$
is a function $f\in C^\infty(G)$ satisfying (A1)--(A3) except possibly $K$--finiteness.
We discuss smooth automorphic forms in more detail the next section.

\medskip
We write $\mathcal A(\Gamma\backslash G)$ (resp., $\mathcal A^\infty(\Gamma\backslash G)$)
for the vector space of all automorphic forms (resp., smooth automorphic forms). Obviously, 
$\mathcal A(\Gamma\backslash G) \subset \mathcal A^\infty(\Gamma\backslash G)$. 
It is easy to see that $\mathcal A(\Gamma\backslash G)$ is a $(\mathfrak g, K)$--module (using \cite{hc},
Theorem 1),
and since $G$ is connected, the space $\mathcal A^\infty(\Gamma\backslash G)$
is $G$--invariant . An automorphic form
$f\in \mathcal A^\infty(\Gamma\backslash G)$  is a $\Gamma$--cuspidal automorphic form if for every proper
$\mathbb Q$--proper parabolic subgroup $\mathcal P\subset \mathcal G$ we have
$$
\int_{U\cap \Gamma \backslash U} f(ux)dx=0, \ \ x\in G,
$$
where $U$ is the group of $\mathbb R$--points of the unipotent radical of  $\mathcal P$.
We remark that the quotient $U\cap \Gamma \backslash U$ is compact. We use normalized $U$--invariant measure on
$U\cap \Gamma \backslash U$. The space of all $\Gamma$--cuspidal automorphic forms (resp.,
$\Gamma$--cuspidal smooth automorphic forms) for $\Gamma$
is denoted by $\mathcal A_{cusp}(\Gamma\backslash G)$ (resp., $\mathcal A^\infty_{cusp}(\Gamma\backslash G)$).
The space $\mathcal A_{cusp}(\Gamma\backslash G)$ is a   $(\mathfrak g, K)$--submodule of $\mathcal A(\Gamma\backslash G)$.
The space $\mathcal A^\infty_{cusp}(\Gamma\backslash G)$ is $G$--invariant.

\vskip .2in 
Following Casselman \cite{casselman-1}, we define
$$
||g||_{\Gamma\setminus G}=\inf_{\gamma\in \Gamma} ||\gamma g||, \ \ g \in G.
$$
It is obvious that $||\cdot ||_{\Gamma\setminus G}$ is $\Gamma$--invariant on the right, and that
$||g||_{\Gamma\setminus G}\le ||g||$ for all $g\in G$. The condition (A-3) is equivalent to

\medskip
\begin{itemize}
\item[(A-3')] there exists $r\in\mathbb R$, $r>0$
  such that for each $u\in \mathcal U(\mathfrak g_{\mathbb C})$ there exists a constant
  $C_u>0$ such that $\left|u.f(x)\right|\le C_u \cdot ||x||^r_{\Gamma\setminus G}$, for all $x\in G$.
\end{itemize}

\vskip .2in
We recall the following standard result:

\begin{Lem}\label{paf-1} Under above assumptions, we have the following:
  \begin{itemize}
         \item[(a)] If  $f\in C^\infty(G)$ satisfies (A-1), (A-2), and
           there exists  $p\ge 1$ such that $f\in L^p(\Gamma\backslash G)$, then f satisfies (A-3), and it is
           therefore an
           automorphic form. We speak about $p$--integrable automorphic form, for $p=1$ (resp., $p=2$) we speak
           about integrable
        (resp., square--integrable) automorphic form.
      \item[(b)] Let $p\ge 1$. Every $p$--integrable automorphic form is integrable.
      \item[(c)] Bounded integrable automorphic form is square--integrable.
      \item[(d)] If $f$ is square integrable automorphic form, then the minimal $G$--invariant closed subspace of
        $L^2(\Gamma\backslash G)$ is a direct is of finitely many irreducible unitary representations.
       \item[(e)] Every $\Gamma$--cuspidal automorphic form is square--integrable.
    \end{itemize}
  \end{Lem}
\begin{proof} For the claims (a) and (e) we refer to \cite{BJ} and reference there.
  Since the volume of $\Gamma \backslash G$ is finite, the claim (b) follows from H\" older inequality (as in
  \cite{MuicMathAnn}, Section 3). The claim (c) is obvious. The claim (d) follows from (\cite{W1}, Corollary
  3.4.7 and Theorem  4.2.1).
\end{proof}

\medskip
In (\cite{MuicSMOOTH}, Proposition 4.7) we  give a simple proof of Lemma \ref{paf-1} (a) using 
results of Casselman \cite{casselman-1} recalled in the next section.

\section{Some Results of Casselman}\label{src}

In this section we  assume that $G$ is a semisimple connected Lie group with finite center.
We assume that $\Gamma$ is a discrete subgroup of $G$. For example, $\Gamma$ could be a congruence subgroup or just
a trivial group.
\medskip
We recall the definition of the Schwartz space
$\mathcal S\left(\Gamma\backslash G\right)$ defined by Casselman 
(\cite{casselman-1}, page 292). It consists of all functions $f\in C^\infty(G)$ satisfying the following
conditions:
\medskip

\begin{itemize}
\item[(CS-1)] $f$ is left--invariant under $\Gamma$ i.e., $f(\gamma x)=f(x)$ for all
  $\gamma\in \Gamma$, $x\in G$;
\item[(CS-2)] $||f||_{u, -n}<\infty$ for all $u\in \mathcal U(\mathfrak g_{\mathbb C})$, and all natural
  numbers $n\ge 1$.
\end{itemize}

\medskip
In above definition, for  $u\in \mathcal U(\mathfrak g_{\mathbb C})$, and a real number $s$, we let
$$
||f||_{u, s}\overset{def}{=} \sup_{x\in G}  ||x||_{\Gamma\setminus G}^{-s} \left|u.f(x)\right|.
$$
Since $||x||_{\Gamma\setminus G} \ge 1$, we have
$$
||f||_{u, s'}\le ||f||_{u, s},
$$
for $s'>s$.

\vskip .2in 
We recall the following result (see \cite{casselman-1}, 1.8 Proposition):

\begin{Prop}\label{src-1} Using above notation, we have the following:
  \begin{itemize}
  \item[(i)]  The Schwartz space $\mathcal S\left(\Gamma\backslash G\right)$ is a Fr\' echet space
    under the seminorms:
    $||\ ||_{u, -n}$, $u\in \mathcal U(\mathfrak g_{\mathbb C})$, $n\in\mathbb Z_{\ge 1}$.
  \item[(ii)] The right regular representation of $G$ on $\mathcal S\left(\Gamma\backslash G\right)$ is a smooth  Fr\' echet
    representation of moderate growth. 
    \end{itemize}
\end{Prop}

\medskip
We recall the definition of representation of moderate growth. Let $(\pi, V)$ be a continuous representation
on the Fr\' echet space   $V$. We say that $(\pi, V)$ is of moderate growth if it is smooth and 
if for any continuous semi-norm $\rho$ there exists an integer
$n$, a constant $C>0$, and another continuous semi-norm $\nu$ such that 
$$
||\pi(g)v||_\rho\le C ||g||^n ||v||_\nu, \ \ g\in G, \ v\in V.
$$
We recall that the semi-norms on a locally convex vector space (for example, a Frech\' et space) $V$ are
constructed via Minkowski functionals.

\vskip .2in
The following definition is from (\cite{casselman-1}, page 295).

\begin{Def}\label{src-2} The space $\mathcal S\left(\Gamma\backslash G\right)'$ of tempered distributions or
  distributions of moderate growth  on $\Gamma \backslash G$ is the
  strong topological dual of $\mathcal S\left(\Gamma\backslash G\right)$. 
  \end{Def}

\medskip

For convenience of the reader,  we recall the definition of a strong topological dual in our particular case.
By general theory, the subset  $B\subset \mathcal S\left(\Gamma\backslash G\right)$ is bounded if for every neighborhood $V$ of
$0$ there exists $s>0$ such that $B\subset tV$, for $t>s$. This definition is not very practical to use.
Again from the general theory (and easy to see directly), $B\subset \mathcal S\left(\Gamma\backslash G\right)$
is bounded if and only if it is bounded in every semi-norm defining topology on $\mathcal S\left(\Gamma\backslash G\right)$
i.e., 
$$
\sup_{f\in B} ||f||_{u, -n} <\infty, \ \ u\in \mathcal U(\mathfrak g_{\mathbb C}), \ \ n\in  \mathbb Z_{\ge 1}.
$$

\medskip
The strong topological dual
$\mathcal S\left(\Gamma\backslash G\right)'$  of $\mathcal S\left(\Gamma\backslash G\right)$ is
the space of continuous functionals on $X$ equipped with strong topology i.e. topology of uniform convergence on bounded sets
in $\mathcal S\left(\Gamma\backslash G\right)$ i.e. topology given by semi--norms
$$
||\alpha||_B=\sup_{f\in B} \ \left|\alpha(f)\right|, \ \ \text{where $B$ ranges over bounded sets of
  $\mathcal S\left(\Gamma\backslash G\right)$}.
$$
By general theory of topological vector spaces, the space $\mathcal S\left(\Gamma\backslash G\right)'$ is a complete
locally convex  (defined by above semi-norms) vector space.

 \medskip
 
The natural action of $G$ on
$\mathcal S\left(\Gamma\backslash G\right)'$ is continuous. The usual representation--theoretic arguments are valid
there (\cite{hc}, Section 2).

\vskip .2in
Following Casselman, we consider the two spaces of functions: the functions of moderate growth
$\mathcal A_{mg}(\Gamma \backslash G)$, and
the functions of uniform moderate growth $\mathcal A_{umg}(\Gamma \backslash G)$. The space
$\mathcal A_{mg}(\Gamma \backslash G)$ consists of the functions $f\in C^\infty(G)$ satisfying the following conditions:

\vskip .2in

\begin{itemize}
\item[(MG-1)] $f$ is left--invariant under $\Gamma$ i.e., $f(\gamma x)=f(x)$ for all
  $\gamma\in \Gamma$, $x\in G$;
\item[(MG-2)] for each $u\in \mathcal U(\mathfrak g_{\mathbb C})$ there exists a constant
  $C_u>0$, $r_u\in\mathbb R$, $r_u>0$  such that $\left|u.f(x)\right|\le C_u \cdot ||x||^{r_u}$, for all $x\in G$.
\end{itemize}

\vskip .2in 
The space $\mathcal A_{umg}(\Gamma \backslash G)$ consists
of the functions $f\in C^\infty(G)$ satisfying the following conditions:

\vskip .2in 
\begin{itemize}
\item[(UMG-1)] $f$ is left--invariant under $\Gamma$ i.e., $f(\gamma x)=f(x)$ for all
  $\gamma\in \Gamma$, $x\in G$;
\item[(UMG-2)] there exists $r\in\mathbb R$, $r>0$
  such that for each $u\in \mathcal U(\mathfrak g_{\mathbb C})$ there exists a constant
  $C_u>0$ such that $\left|u.f(x)\right|\le C_u \cdot ||x||^r$, for all $x\in G$.
\end{itemize}
We note that in the second definition $r$ is independent of $u\in \mathcal U(\mathfrak g_{\mathbb C})$.

\medskip
\begin{Lem}\label{src-3} We maintain the assumptions of the first paragraph of Section \ref{paf}.
  Then, the spaces of functions which are $\cal Z(\mathfrak g_{\mathbb C})$--finite and $K$--finite on the right in
  $\mathcal A_{mg}(\Gamma \backslash G)$, and in $\mathcal A_{umg}(\Gamma \backslash G)$ coincide,
  and are equal to the  space $\mathcal A(\Gamma\backslash G)$ of automorphic forms for $\Gamma$. Next,
  the space of smooth automorphic forms $\mathcal A^\infty(\Gamma\backslash G)$ is a subspace of
  $\cal Z(\mathfrak g_{\mathbb C})$--finite functions in $\mathcal A_{umg}(\Gamma \backslash G)$.
  Furthermore, we have
  $$
  \mathcal A(\Gamma\backslash G)\subset \mathcal A^\infty(\Gamma\backslash G)\subset
  \mathcal A_{umg}(\Gamma \backslash G)\subset \mathcal A_{mg}(\Gamma \backslash G).
  $$
  \end{Lem}
\begin{proof} This is a simple observation made in (\cite{MuicSMOOTH}, Lemma 4.4).
\end{proof}

\vskip .2in

\begin{Lem}\label{src-4} The Garding space in   $\mathcal S\left(\Gamma\backslash G\right)'$ is equal to the space
  $\mathcal A_{umg}(\Gamma \backslash G)$.
  \end{Lem}
\begin{proof} This (\cite{casselman-1}, Theorem 1.16).
\end{proof}

\vskip .2in
We remark that $\mathcal S\left(\Gamma\backslash G\right)'$ is not a Fr\' echet space  so \cite{dixmal} can not be applied
to prove that the space of smooth vectors is the same as the Garding space. Therefore, for example, in the settings of
Lemma \ref{src-3},
$\mathcal A^\infty(\Gamma\backslash G)$ is just subspace of the space of all $\cal Z(\mathfrak g_{\mathbb C})$--finite vectors in
$\mathcal S\left(\Gamma\backslash G\right)'$.

\vskip .2in
Regarding smooth vectors in $\mathcal S\left(\Gamma\backslash G\right)'$, the following 
lemma will be used later (see \cite{MuicSMOOTH}, Lemma 4.6):

\medskip 
\begin{Lem}\label{src-5} Assume that $f\in L^p(\Gamma\setminus G)$, for some $p\ge 1$, and $\alpha\in C_c^\infty(G)$. Then,
  $f\star \alpha$ is equal almost everywhere to a function in  $\mathcal A_{umg}(\Gamma \backslash G)$.
\end{Lem}

\section{Some Results on the Spaces  $\mathcal S(\Gamma \backslash G)'$}\label{srs}

This section is strongly motivated by a lecture of Wallach \cite{W3}. Some of the results here
are probably well--known, and we present our way of understanding them. We also give a complete description of irreducible closed
subrepresentations  $\mathcal S(\Gamma\backslash G)'$. We prove that under proper assumptions on $G$ and $\Gamma$ only finite dimensional
subrepresentation of  $\mathcal S(\Gamma\backslash G)'$ is trivial representation. 
\vskip .2in 

In this section, we
let $(\pi, \mathcal H)$ be an irreducible admissible representation of $G$  acting on the Hilbert space $\mathcal H$.
 We write $\langle \ , \ \rangle$ for the inner product on $\mathcal H$. We denote by $\mathcal H^\infty$ 
the subspace of smooth vectors in 
$\mathcal H$. It is a complete Fr\' echet space under the family of semi--norms:

$$
||h||_u =||\pi(u)h||, \ \ u\in \mathcal U(\mathfrak g_{\mathbb C}),
$$
where $||\ ||$  is the norm on $\mathcal H$ derived from  $\langle \ , \ \rangle$. It is a smooth Frech\' et representation of moderate growth
(\cite{W2}, Lemma 11.5.1). In particular, if $\lambda$ is a continuous functional on $\mathcal H^\infty$, then there exists
$d\in \mathbb R$, and a continuous semi-norm $\kappa$ such that
\begin{equation}\label{srs-01}
\left|\lambda\left(\pi(g)h\right)\right|\le ||g||^d \kappa(h), \ \ g\in G, \ \ h\in \mathcal H^\infty.
\end{equation}
The reader can easily check that if (\ref{srs-01}) holds for any $d=d_0$, then it holds for
all $d\ge d_0$. We make the following definition (see also \cite{MuicJFA}, (3-4)):

\medskip 
\begin{Def}\label{srs-d1} Let $d_{\mathcal H, \lambda}=d_{\pi, \lambda}\ge -\infty$ be the infimum of all 
$d\in \mathbb R$ 
such that (\ref{srs-01}) holds for some continuous semi--norm $\kappa=\kappa_d$. 
  \end{Def}

\medskip
\begin{Lem}\label{srs-02} The Frech\' et representation $G$ on $\mathcal H^\infty$  is irreducible in the category of
  Frech\' et representations.
\end{Lem}
\begin{proof} This representation is a canonical globalization (see \cite{W2}, Chapter 11, or \cite{casselman}) of a
  $(\mathfrak g, K)$--module $\mathcal H_K$. Hence, the lemma. It is also to give a direct proof. Let $\mathcal V\subset \mathcal
  H^\infty$ be a closed subrepresenation different than $\{0\}$. Pick any $v\in \mathcal V$, $v\ne 0$. Then since $\mathcal H^\infty$ is a
  smooth representation, the Fourier expansion converges absolutely (\cite{hc}, Lemma 5):
  $$
  v=\sum_{\delta\in \hat{K}} \ E_\delta(v),
  $$
  where we fix the normalized Haar
measure $dk$ on $K$, and  let
$$
E_\delta(v)=\int_{K} d(\delta) \overline{\xi_\delta(k)}\ \pi(k)v\ dk. 
$$
Here, as usual $\hat{K}$ is the set of
equivalence of irreducible representations of $K$. Also, for $\delta\in 
\hat{K}$, we write $d(\delta)$ and $\xi_\delta$ for the degree and
character of $\delta$, respectively.
The vector $E_\delta(v)$ belongs to the $\delta$--isotypic component $\mathcal V(\delta)$ of
$\mathcal V$.  This shows that $\mathcal H_K\cap \mathcal V$ is dense in $\mathcal V$. In particular,
$\mathcal H_K\cap \mathcal V$ is non--zero $(\mathfrak g, K)$--submodule of $\mathcal H_K$. Hence,
$\mathcal H_K\subset \mathcal V$ since $\mathcal H_K$ is irreducible. But because of the same reason
$\mathcal H_K$ is dense in $\mathcal H^\infty$. This implies that $\mathcal V=\mathcal H^\infty$.
 \end{proof}

\vskip .2in

\begin{Prop}\label{srs-0}
Let $\Gamma\subset G$ be a discrete subgroup.
Let $\lambda$ be a continuous functional on $\mathcal H^\infty$ which is $\Gamma$--invariant.  Then, we have the following:
\begin{itemize}
\item[(i)] The  pairing
$
\mathcal H^\infty \times \mathcal S(\Gamma\backslash G)\longrightarrow \mathbb C
$
given by $(h, f)\longmapsto \int_{\Gamma\backslash G}\lambda(\pi(g)h) f(g)dg$
is well--defined, continuous, and $G$--equivariant. 

\item[(ii)]  The map $ \mathcal H^\infty \longrightarrow \mathcal S(\Gamma \backslash G)'$
which maps $h\longmapsto \alpha_{\lambda, \Gamma}(h)$
where
$$
\alpha_{\lambda, \Gamma}(h)(f)=\int_{\Gamma\backslash G} \lambda(\pi(g)h) f(g)dg, \ \ f\in \mathcal S(\Gamma\backslash G),
$$
is a continuous map of locally convex representations of $G$. The image is contained in $\mathcal A_{umg}(\Gamma \backslash G)$.
\item[(iii)] If $\lambda\neq 0$, then $\alpha_{\lambda, \Gamma}$ is an embedding. The closure 
$Cl\left(\alpha_{\lambda, \Gamma}\left(\mathcal H^\infty\right) \right)$ is a closed irreducible admissible subrepresentation of 
  $\mathcal S(\Gamma\backslash G)'$. 
\end{itemize}
\end{Prop}
\begin{proof} We prove (i). First, we may assume that $d>0$ in  (\ref{srs-01}). Then, 
$\Gamma$--invariance implies that
$$
\left|\lambda\left(\pi(g)h\right)\right|=\left|\lambda\left(\pi(\gamma g)h\right)\right|\le ||\gamma g||^d \kappa(h), 
$$
for all $\gamma\in \Gamma$,  $g\in G$, and $h\in \mathcal H^\infty$.  Hence
$$
\left|\lambda\left(\pi(g)h\right)\right| \le ||g||^d_{\Gamma\backslash G} \kappa(h), 
$$
 $g\in G$, and $h\in \mathcal H^\infty$.

Next,  $\int_G||g|^{-d_0}dg<\infty$ for all sufficiently large $d_0>0$. Then,
(\cite{casselman-1}, Proposition 1.9) implies that $\int_{\Gamma\backslash G}||g|^{-d_0}_{\Gamma\backslash G} dg<\infty$ for
all sufficiently large $d_0>0$. Hence
\begin{equation}\label{srs-03}
\left| \int_{\Gamma\backslash G}\lambda(\pi(g)h) f(g)dg\right|\le  \int_{\Gamma\backslash G}\left|\lambda(\pi(g)h) f(g)\right|dg \le
\kappa(h)||f||_{1,-d_0} \cdot \int_{\Gamma\backslash G}\frac{1}{||g||^{-d+d_0}_{\Gamma\backslash G}} dg.
\end{equation}
Consequently, the pairing is well--defined and continuous. It is clearly $G$--equivariant.
This proves (i).

Now, we prove (ii). The continuity of $\alpha_{\lambda, \Gamma}$ is obvious from above inequality since if 
$B\subset \mathcal S(\Gamma \backslash G)$ is bounded, and  if we let
$$
M_B=\sup_{f\in B} \ ||f||_{1,-d_0}<\infty,
$$
then we have
$$
||\alpha_{\lambda, \Gamma}(h)||_B=\sup_{f\in B} \ |\alpha_{\lambda, \Gamma}(h)(f)|\le M_B \cdot \left( \int_{\Gamma\backslash G}\frac{1}{
||g||^{-d+d_0}_{\Gamma\backslash G}} dg \right) \kappa(h), \ \
h\in \mathcal H^\infty.
$$
Next, the first paragraph of the proof shows that  the function $g\longmapsto \lambda(\pi(g)h)$ belongs to 
$ \mathcal A_{umg}(\Gamma \backslash G)$.  This completes the proof of (ii).

The different argument is based on results of Casselman (see Lemma \ref{src-4}).
Indeed. becasue of the Dixmier--Malliavin, each $h\in \mathcal H^\infty$ can be written in the form
$$
h=\sum_{i=1}^l \pi(\beta_i)h_i,
$$
for some $\beta_i\in C_c^\infty(G)$ and $h_i\in \mathcal H^\infty$.  Hence, we have
$$
\alpha_{\lambda, \Gamma}(h)=\sum_{i=i}^l r'(\beta_i)\alpha(h_i)
$$
which implies that $\alpha_{\lambda, \Gamma}(h)\in \mathcal A_{umg}(\Gamma\backslash G)$.

Now, we prove (iii). Let  $f\in C_c^\infty(G)$. Then, $P_\Gamma(f)(x)\overset{def}{=}\sum_{\gamma\in \Gamma} f(\gamma x)$ 
for $x\in G$, defines an element of  $\mathcal S(\Gamma\backslash G)$ which is compactly supported modulo $\Gamma$.
For $h\in \mathcal H^\infty$, we have 
$$
\alpha_{\lambda, \Gamma}(h)\left(P_\Gamma(f)\right)=
\int_{\Gamma\backslash G} \lambda(\pi(g)h) P_\Gamma(f)(g)dg=
\int_G \lambda(\pi(g)h) f(g)dg.
$$
Letting $f\in C_c^\infty(G)$ vary, we see that there exists at least one  $h\in \mathcal H^\infty$ such that 
$\alpha_{\lambda, \Gamma}(h)\neq 0$ provided that $\lambda\neq 0$. In view of Lemma \ref{srs-02}, this implies that 
$\alpha_{\lambda, \Gamma}$ is an embedding. Next, as in the proof of Lemma \ref{srs-02}, we define projectors
$$
E_\delta(\alpha)=\int_{K} d(\delta) \overline{\xi_\delta(k)}\ r'(k)\alpha\ dk, \ \ \alpha\in  \mathcal S(\Gamma\backslash G)'
$$
for $\delta\in \hat{K}$. Since $\alpha_{\lambda, \Gamma}\left(\mathcal H^\infty\right)$ is obviously dense in 
$Cl\left(\alpha_{\lambda, \Gamma}\left(\mathcal H^\infty\right) \right)$, we have that 
$$
E_\delta\left(\alpha_{\lambda, \Gamma}\left(\mathcal H^\infty\right)\right)
$$
is dense in
$$
E_\delta\left(Cl\left(\alpha_{\lambda, \Gamma}\left(\mathcal H^\infty\right) \right)
\right).
$$
But
$$
E_\delta\left(\alpha_{\lambda, \Gamma}\left(\mathcal H^\infty\right)\right)=
\alpha_{\lambda, \Gamma}\left(E_\delta\left(\mathcal H^\infty\right)\right)=\alpha_{\lambda, \Gamma}\left(
\mathcal H^\infty(\delta)\right)=\alpha_{\lambda, \Gamma}\left(
\mathcal H_K(\delta)\right)
$$
is a finite--dimensional space. Hence, it is closed. Thus, we have that
$$
Cl\left(\alpha_{\lambda, \Gamma}\left(\mathcal H^\infty\right)\right) (\delta)\overset{def}{=}
E_\delta\left(Cl\left(\alpha_{\lambda, \Gamma}\left(\mathcal H^\infty\right) \right)
\right)=\alpha_{\lambda, \Gamma}\left(\mathcal H_K(\delta)\right)
$$
is finite--dimensional. This proves that $Cl\left(\alpha_{\lambda, \Gamma}\left(\mathcal H^\infty\right)\right) $
is admissible. We show that $Cl\left(\alpha_{\lambda, \Gamma}\left(\mathcal H^\infty\right)\right) $ is irreducible i.e.,
only closed $G$--invariant subspaces of $Cl\left(\alpha_{\lambda, \Gamma}\left(\mathcal H^\infty\right)\right) $
are $\{0\}$ and $Cl\left(\alpha_{\lambda, \Gamma}\left(\mathcal H^\infty\right)\right) $.  We use smooth vectors.

Using the argument from (\cite{W1}, Lemma 1.6.4), the subspace of smooth vectors
$Cl\left(\alpha_{\lambda, \Gamma}\left(\mathcal H^\infty\right)\right)^\infty$ in
$Cl\left(\alpha_{\lambda, \Gamma}\left(\mathcal H^\infty\right)\right) $
is a complete locally convex representation of $G$ where topology is defined by the semi--norms:
$$
\alpha\longmapsto ||r'(u)\alpha||_B,
$$
where $u\in  \mathcal U(\mathfrak g_{\mathbb C})$ and $B\subset \mathcal S(\Gamma \backslash G)$ is bounded. 
The key thing is that each smooth vector has a Fourier expansion analogous to the one in the proof of Lemma \ref{srs-02}.
Then, as in the proof of Lemma \ref{srs-02} we see that  $Cl\left(\alpha_{\lambda, \Gamma}\left(\mathcal H^\infty\right)\right)^\infty$
is irreducible meaning that only $G$--invariant subspaces are trivial and everything.

Now, if $W\subset Cl\left(\alpha_{\lambda, \Gamma}\left(\mathcal H^\infty\right)\right)$ is closed $G$--invariant subspace. Assume $W\neq 0$. Then
$$
W^\infty\subset Cl\left(\alpha_{\lambda, \Gamma}\left(\mathcal H^\infty\right)\right)^\infty
$$
is closed $G$--invariant subspace in appropriate topology. It is dense in $W$ (see \cite{hc}, Corollary 1),
and therefore non--zero. But then we must
have
$$
W^\infty=Cl\left(\alpha_{\lambda, \Gamma}\left(\mathcal H^\infty\right)\right)^\infty.
$$
Again because the smooth vectors are dense (\cite{hc}, Corollary 1), this implies
$$
W= Cl\left(\alpha_{\lambda, \Gamma}\left(\mathcal H^\infty\right)\right)
$$
\end{proof}

\vskip .2in 
 The Garding space $\mathcal A_{umg}(\Gamma \backslash G)$ has a natural filtration by the smooth Frech\' et representations:
 $$
 \mathcal S(\Gamma \backslash G)\subset \cdots
 \subset \mathcal A_{umg, -1}(\Gamma \backslash G) \subset \mathcal A_{umg, 0}(\Gamma \backslash G)\subset
 \mathcal A_{umg, 1}(\Gamma \backslash G)\subset
 \mathcal A_{umg, 2}(\Gamma \backslash G)\subset \cdots,
 $$
 where for an integer  $n$ we let
  $$
 \mathcal A_{umg, n}(\Gamma \backslash G)=\left\{\varphi\in \mathcal A_{umg}(\Gamma \backslash G); \ \
 ||\varphi||_{u, n}<\infty, \ \ u\in
 \mathcal U(\mathfrak g_{\mathbb C})\right\}.
 $$
 We remark that all embeddings are continuous, and that
 $$
 \mathcal S(\Gamma \backslash G)=\cap_{n\in \mathbb Z} \  \mathcal A_{umg, n}(\Gamma \backslash G)
 $$
 We may therefore let
  $$
 \mathcal A_{umg, -\infty}(\Gamma \backslash G)= \mathcal S(\Gamma \backslash G).
 $$
 \medskip

 \begin{Lem}\label{srs-24} Let $n\ge -\infty$. Then, the representation of $G$  on $\mathcal A_{umg, n}(\Gamma \backslash G)$
   is of moderate growth.
 \end{Lem}
 \begin{proof} The proof is similar to the proof of (\cite{W2}, Lemma 11.5.1).
   We remarked above that representation is smooth. Let $\rho$ be the continuous seminorm on
   $\mathcal A_{umg, n}(\Gamma \backslash G)$. Then, since $\rho$ is continuous,
   there exists a constant $C>0$, and  $u_1, \ldots, u_l \in \mathcal U(\mathfrak g_{\mathbb C})$ such that 
   $$
   ||\varphi||_\rho\le C\cdot \left(  ||\varphi||_{u_1, n}+\cdots +  ||\varphi||_{u_l, n}\right), \ \ \varphi\in
   \mathcal A_{umg, n}(\Gamma \backslash G).
   $$
   This is for the case $n>-\infty$. But when $n=-\infty$, the above is true for convenient integer (again denoted by) $n$.
   In this case, we fix such $n$.

   Next, we consider the standard filtration of $\mathcal U(\mathfrak g_{\mathbb C})$ by finite $G$--invariant subspaces:
   $$
   \mathcal U^0(\mathfrak g_{\mathbb C})=\mathbb C \subset \mathcal U^1(\mathfrak g_{\mathbb C}) \subset \mathcal U^2(\mathfrak g_{\mathbb C}) \subset
   \cdots.
   $$

   Let $k\ge 0$. Let $v_1, \ldots v_k$ be the basis of $\mathcal U^k(\mathfrak g_{\mathbb C}) $. Then, there exists smooth functions
   $\eta_{i, j}$ such that
   $$
   Ad(g)v_i=\sum_{j=1}^k \eta_{ij}(g) v_j.
   $$
   Clearly, $\eta_{i, j}$ are matrix coefficients of the representation on   $\mathcal U^k(\mathfrak g_{\mathbb C}) $.
   By the construction of the norm,   there exists $D, r>0$ such that
   $$
   |\eta_{ij}(g)|\le D \cdot ||g||^r, \ \ g\in G,
   $$
   for all $i, j$. 

   We assume that $k$ is large enough so that $u_1, \ldots, u_l \in \mathcal U^k(\mathfrak g_{\mathbb C})$. Then, we can write
   $$
   Ad(g)u_i=\sum_{j=1}^k \nu_{ij}(g) v_j.
   $$
   The functions $\nu_{ij}$ are linear combinations of functions $\eta_{ij}$. Therefore,  there exists $D_1>0$ such that
   $$
   |\nu_{ij}(g)|\le D_1 \cdot ||g||^r, \ \ g\in G,
   $$
   for all $i, j$.

   Now, for $\varphi\in
   \mathcal A_{umg, n}(\Gamma \backslash G)$, and $g\in G$, using  properties of the norm, we have
   \begin{align*}
     ||r(g)\varphi||_\rho& \le C\cdot \sum_{i=1}^l||\varphi||_{u_i, n}= C\cdot \sum_{i=1}^l \sup_{x\in G} \ ||x||^{-n} \left|u_i. r(g)\varphi(x)\right|\\
     &= C\cdot \sum_{i=1}^l \sup_{x\in G} \ ||x||^{-n} \left|r(g) Ad(g^{-1}u_i).\varphi(x)\right|\\
     &= C\cdot \sum_{i=1}^l \sup_{x\in G} \ ||x||^{-n} \left|(Ad(g^{-1})u_i).\varphi(xg)\right|\\
     &\le C\cdot \sum_{i=1}^l \sum_{j=1}^k \left|\nu_{ij}(g^{-1})\right|  \sup_{x\in G} \ ||x||^{-n} \left|v_j.\varphi(xg)\right|\\
     &=C\cdot \sum_{i=1}^l \sum_{j=1}^k \left|\nu_{ij}(g^{-1})\right|  \sup_{x\in G} \ ||xg^{-1}||^{-n} \left|v_j.\varphi(x)\right|\\
     &\le CD_1\cdot \sum_{i=1}^l \sum_{j=1}^k ||g^{-1}||^r  \sup_{x\in G} \ ||xg^{-1}||^{-n} \left|v_j.\varphi(x)\right|\\
     &\le CD_1\cdot \sum_{i=1}^l \sum_{j=1}^k ||g||^{n+r}  \sup_{x\in G} \ ||x||^{-n} \left|v_j.\varphi(x)\right|\\
      &= l CD_1 ||g||^{n+r}   \sum_{j=1}^k  ||\varphi||\\
     \end{align*}
     \end{proof}

 \medskip
 \begin{Lem}\label{srs-22} Let $n\ge -\infty$. Then, the linear functional $\varphi\longmapsto \varphi(1)$ is
   continuous on
   $\mathcal A_{umg, n}(\Gamma \backslash G)$. 
 \end{Lem}
 \begin{proof} Assume first that $n>-\infty$. Then, we have
   $$
   |\varphi(1)|=||1|^{-n}_{\Gamma \backslash G} |\varphi(1)|\le ||\varphi||_{1, -n},
   $$
   for 
   $\varphi\in \mathcal A_{umg, n}(\Gamma \backslash G)$. The case $n=-\infty$ is a consequence of above inequalities.
   \end{proof}

 \medskip

 \begin{Lem}\label{srs-23} 
   Let $\Gamma\subset G$ be a discrete subgroup.
   Let $\lambda$ be a continuous functional on $\mathcal H^\infty$ which is $\Gamma$--invariant.
   For any integer $n$ such that $n> d_{\pi, \lambda}$,   the map which assigns  to
   $h\in\mathcal H^\infty$ a function 
 $g\longmapsto \lambda(\pi(g)h)$ in  $\mathcal A_{umg, n}(\Gamma \backslash G)$
  is continuous, $G$--equivariant, and if $\lambda\neq 0$, then  it is an embedding. Moreover, the same holds if
   $n=d_{\pi, \lambda}=-\infty$.
   \end{Lem}
 \begin{proof} By definition of $d_{\pi, \lambda}$ (see Definition \ref{srs-d1}) and the fact that $||x||\ge 1$ for all $x\in G$, we have
   (see (\ref{srs-01}))
   $$
  \left|\lambda\left(\pi(g)h\right)\right|\le ||g||^d \kappa(h), \ \ g\in G, \ \ h\in \mathcal H^\infty.
  $$
  The semi--norm $h\longmapsto \kappa(\pi(u)h)$ is again continuous, for  $u\in \mathcal U(\mathfrak g_{\mathbb C})$, and we have
  as a consequence of above inequality
  $$
  \left|\lambda\left(\pi(g)\pi(u)h\right)\right|\le ||g||^d \kappa(\pi(u)h), \ \ g\in G, \ \ h\in \mathcal H^\infty.
  $$
  For $\gamma\in \Gamma$, the $\Gamma$--invariance of $\lambda$ implies that
  $$
  \left|\lambda\left(\pi(g)\pi(u)h\right)\right|=\left|\lambda\left(\pi(\gamma g)\pi(u)h\right)\right|
  \le ||\gamma g||^d \kappa(\pi(u)h)
  $$
  Since the norm is continuous and $\Gamma$ discrete, for fixed $x\in G$, there exists $\gamma_0\in \Gamma$ such that
  $$
  ||\gamma_0 x||=||x||_{\Gamma \backslash G}=\inf_{\gamma\in \Gamma} ||\gamma x||.
  $$
  Thus above inequality implies
$$
  \left|\lambda\left(\pi(g)\pi(u)h\right)\right|
  \le ||g||^d_{\Gamma\backslash G} \kappa(\pi(u)h).
  $$
  This implies that
  $$
 \sup_{g\in G} \  ||g||^{-d}_{\Gamma\backslash G}\left|\lambda\left(\pi(g)\pi(u)h\right)\right|
  \le \kappa(\pi(u)h).
  $$
  Now, the lemma easily follows.
     \end{proof}

 \vskip .2in
 Now, we prove the main result of the present section.
 \vskip .2in
 
\begin{Thm}\label{srs-29}
  Let $\mathcal V \subset \mathcal S(\Gamma\backslash G)'$ be a closed irreducible admissible subrepresentation of
  $\mathcal S(\Gamma\backslash G)'$. Then, there exists an irreducible admissible representation of $G$  acting on the Hilbert space
  $\mathcal H$, and  a non--zero $\Gamma$--invariant continuous functional on $\mathcal H^\infty$ such that
  $$
  \mathcal V= Cl\left(\alpha_{\lambda, \Gamma}\left(\mathcal H^\infty\right) \right).
  $$
\end{Thm}
\begin{proof} By (\cite{hc}, Lemma 4), $\mathcal V^\infty \cap \mathcal V_K$ is dense in $\mathcal V$. Since $\mathcal V$ is admissible, we see
  that $\mathcal V_K\subset \mathcal V^\infty$. It is easy to check that $ \mathcal V_K$  is an irreducible $(\mathfrak g, K)$--module.
  In particular, every vector in $ \mathcal V_K$ is $\mathcal Z(\mathfrak g_{\mathbb C})$--finite. Therefore, by (\cite{hc}, Theorem 1), for each
  $\varphi\in \mathcal V_K$ there exists $\alpha\in C_c^\infty(G)$ such that $r'(\alpha)\varphi=\varphi$. Hence, $\mathcal V_K$ belongs to the
  Garding space of $\mathcal V$, and consequently to the Garding space of $\mathcal S(\Gamma\backslash G)'$ which is
  $\mathcal A_{umg}(\Gamma\backslash G)$. By means of the Casselman subrepresentation theorem, we can find an infinitesimal embedding  of
  $\mathcal V_K$ into a principal series of $G$. In this way, we obtain a globalization of $\mathcal V_K$ i.e., there exists an irreducible
  admissible representation $(\pi, \mathcal H)$ on the Hilbert space $\mathcal H$ infinitesimally equivalent to    $\mathcal V_K$. Let us fix
  an isomorphism $\eta: \mathcal H_K\longrightarrow \mathcal V_K$.

  We recall the filtration of $\mathcal A_{umg}(\Gamma\backslash G)$ by the representations of moderate growth (see Lemma \ref{srs-24}) :
  $$
  \mathcal A_{umg, 1}(\Gamma\backslash G)\subset \mathcal A_{umg, 2}(\Gamma\backslash G)\subset
  \mathcal A_{umg,3}(\Gamma\backslash G) \subset \cdots.
  $$
  This is also filtration of   $\mathcal U(\mathfrak g_{\mathbb C})$--modules. Since $\mathcal V_K$ is irreducible, there exists $n\ge 1$
  such that
  $$
  \mathcal V_K \subset   \mathcal A_{umg, n}(\Gamma\backslash G).
  $$
  Let $V_n$ be the closure of $\mathcal V_K$ in  $\mathcal A_{umg, n}(\Gamma\backslash G)$. It is obvious that
  a $(\mathfrak g, K)$--module on the space of
 $K$--finite vectors  in $W_n$ is $\mathcal V_K$. Therefore, $V_n$ is irreducible.
 We remark that $V_n$  being a closed subrepresentation of a
 representation of moderate growth $\mathcal A_{umg, n}(\Gamma\backslash G)$ is also a representation of moderate growth
 (see Lemma \ref{srs-24}; \cite{W2}, Lemma 11.5.2). But $\mathcal H^\infty$
 is also a representation of moderate growth (\cite{W2}, Lemma 11.5.1) and irreducible (see Lemma \ref{srs-02}).
 So,  the isomorphism  $\eta: \mathcal H_K\longrightarrow \left(W_n\right)_K=\mathcal V_K$, extends to a continuous isomorphism
 of $G$--representations $\eta: \mathcal H^\infty\longrightarrow V_n$ applying (\cite{W2}, Theorem 11.5.1).

 Now, the required linear functional is
 $$
 \lambda(h)\overset{def}{=}\eta(h)(1).
 $$
 Indeed, it is obviously continuous (see Lemma \ref{srs-22}). Next, it is $\Gamma$--invariant since
 $$
 \lambda(\pi(\gamma)h)=\eta(\pi(\gamma)h)(1)=r(\gamma)\eta(h)(1)=\eta(h)(\gamma)=\eta(h)(1), \ \ h\in\mathcal H^\infty, \ \ \gamma\in \Gamma.
 $$
 Now, using the notation introduced in Proposition \ref{srs-0}, we compute
 $$
 \lambda(\pi(g)h)=\eta(\pi(g)h)(1)=r(g)\eta(h)(1)=\eta(h)(g),  \ \ h\in\mathcal H^\infty, \ \ g \in G.
$$ 
 For $h\in \mathcal H_K$, above computation and Proposition \ref{srs-0} (ii) implies that
 $$
 \alpha_{\lambda, \Gamma}\left(\mathcal H_K\right)=\mathcal V_K.
 $$
 The proof of Proposition \ref{srs-0} shows that the space of $K$--finite vectors of
 $Cl\left(\alpha_{\lambda, \Gamma}\left(\mathcal H^\infty\right) \right)$  is $ \alpha_{\lambda, \Gamma}\left(\mathcal H_K\right)$ and it is dense in
 $Cl\left(\alpha_{\lambda, \Gamma}\left(\mathcal H^\infty\right) \right)$. Since $\mathcal V_K$ is dense in
 $\mathcal V_K$, the theorem follows.
  \end{proof}

\vskip .2in
Examples of subrepresentation can explicitly be constructed using Eisenstein series \cite{langlands},  or be shown to exists using
Poincar\' e series \cite{MuicSMOOTH}, or the trace formula (\cite{arthur}, \cite{arthur-1}). 
Now, we show that there are no finite--dimensional representations except the trivial representation in
$\mathcal S(\Gamma\backslash G)'$ under appropriate assumptions.
 \vskip .2in
 
 \begin{Thm}\label{srs-30} We maintain Assumption \ref{paf-00}. Then, if $G$ has no compact components, then the trivial representation
   is the only finite--dimensional subrepresentation of $\mathcal S(\Gamma\backslash G)'$.
   \end{Thm}
 \begin{proof} Let $\mathcal V\subset \mathcal S(\Gamma\backslash G)'$ be a finite--dimensional subrepresentation. Then, by Theorem \ref{srs-29},
   there exists finite dimensional representation $\mathcal H$ (satisfying the assumption of the first paragraph of this section) such that
   $$
  \mathcal V= Cl\left(\alpha_{\lambda, \Gamma}\left(\mathcal H^\infty\right) \right)= \alpha_{\lambda, \Gamma}\left(\mathcal H^\infty\right).
  $$
  Since  $\mathcal H$ is finite--dimensional, we have  $\mathcal H^\infty=\mathcal H$, and there is a non--zero $\Gamma$--invariant functional
  on $\mathcal H$. So, the algebraic dual $\mathcal H'$ is smooth irreducible representation of $G$ having a non--zero $\Gamma$--invariant
  vector. By  general theory, $\mathcal H'$   is a restriction of an algebraic (holomorphic) representation of $\mathcal G(\mathbb C)$ to $G$.
  But the Borel density theorem \cite{Borel1960} implies that $\Gamma$ is Zariski dense in $\mathcal G(\mathbb C)$. Because of that
  a $\Gamma$--invariant vector is also $\mathcal G(\mathbb C)$--invariant. In particular, it is $G$--invariant. But $\mathcal H'$ is
  an irreducible representation of $G$. Hence, $\mathcal H'$ is one--dimensional and $G$ acts trivially. Thus, the same holds for
  $\mathcal H$ and consequently for $\mathcal V$. 
 \end{proof}

 \vskip .2in
 The most important consequence of Theorem \ref{srs-30} is the following corollary:
 
 \begin{Cor}\label{srs-31} We maintain Assumption \ref{paf-00}. Then, if $G$ has no compact components, then the trivial representation
   is the only finite--dimensional subrepresentation of a $(\mathfrak g, K)$--module $\mathcal A(\Gamma\backslash G)$ (defned in Section
   \ref{paf}).
   \end{Cor}

\section{Results on $L^2(\Gamma \backslash G)$}\label{srs-nas}

\medskip

In this section we continue with the assumptions of previous Section  \ref{srs}. The reader should review 
the second paragraph of Section \ref{srs}.

\medskip

We consider the usual embedding $\mathcal A_{umg}(\Gamma\backslash G)\hookrightarrow \mathcal S(\Gamma\backslash G)'$, given by
$\varphi\longmapsto \beta_\varphi$ where $\beta_\varphi$ is defined by $\beta_\varphi(f)=\int_{\Gamma \backslash G}
\varphi(x)f(x)dx$, for $f\in \mathcal S(\Gamma\backslash G)$.

\vskip .2in
\begin{Lem}\label{srs-28} We equip the space of smooth vectors $\left(\mathcal S(\Gamma\backslash G)'\right)^\infty$
  with the usual  topology (described in the proof below). Let $n\ge -\infty$. Then, the 
   embedding $\mathcal A_{umg, n}(\Gamma \backslash G)\hookrightarrow \left(\mathcal
   S(\Gamma\backslash G)'\right)^\infty$,
   given by $\varphi\longmapsto \beta_\varphi$,  is $G$--equivariant and continuous.
 \end{Lem}
\begin{proof} Recall that  the space of smooth vectors $\left(\mathcal S(\Gamma\backslash G)'\right)^\infty$ in
  $\mathcal S(\Gamma\backslash G)'$
  is a complete locally convex representation of $G$ where topology is defined by the semi--norms:
$$
\alpha\longmapsto ||r'(u)\alpha||_B,
$$
where $u\in  \mathcal U(\mathfrak g_{\mathbb C})$ and $B\subset \mathcal S(\Gamma \backslash G)$ is bounded.

Let $u\in  \mathcal U(\mathfrak g_{\mathbb C})$ and let $B\subset \mathcal S(\Gamma \backslash G)$ be a  bounded set.
Let $\varphi\in \mathcal A_{umg, n}(\Gamma \backslash G)$. Then, assuming that in the computation below $n$ means any integer if
originally we have $n=-\infty$,
$$
\sup_{f\in B} \left|r'(u)\beta_\varphi(f)\right|=\sup_{f\in B} \left|\int_{\Gamma\backslash G} u.\varphi(x)f(x)dx\right|
$$
We note again (\cite{casselman-1}, Proposition 1.9) implies that $\int_{\Gamma\backslash G}||g|^{-d_0}_{\Gamma\backslash G} dg<\infty$ for
all sufficiently large $d_0>0$.  Let $M_B=\sup_{f\in B} \ ||f||_{1,-d_0}<\infty$. Hence
$$
\sup_{f\in B} \left|r'(u)\beta_\varphi(f)\right| \le \left(M_B \cdot \int_{\Gamma\backslash G}||g|^{n-d_0}_{\Gamma\backslash G} dg\right)
||u.\varphi||_{u, n}. 
$$
The continuity of the map easily follows. The map is obviously $G$--equivariant.
\end{proof}

\vskip .2in

  We recall the classical and well--known argument in our settings. 
Let $X\in \mathfrak g$. Then, for $F\in L^1(\Gamma \backslash G)\cap C^\infty(G)$ and $f\in \mathcal
S(\Gamma\backslash G)$, we have
\begin{align*}
  &\int_{\Gamma\backslash G} X.F(x)f(x)dx=  \\
  &=
\int_{\Gamma\backslash G} \frac{d}{dt}|_{t=0} F(x\exp{(tX)})f(x)dx\\
&=
\int_{\Gamma\backslash G} \frac{d}{dt}|_{t=0}\left( F(x\exp{(tX)}) f(x\exp{(tX)})\right)dx-
\int_{\Gamma\backslash G} F(x) \frac{d}{dt}|_{t=0} f(x\exp{(tX)})dx\\
&=
\int_{\Gamma\backslash G} \frac{d}{dt}|_{t=0}\left( F(x) f(x)\right)dx-
\int_{\Gamma\backslash G} F(x) \frac{d}{dt}|_{t=0} f(x\exp{(tX)})dx\\
&=- \int_{\Gamma\backslash G} F(x) \frac{d}{dt}|_{t=0} f(x\exp{(tX)})dx.
\end{align*}

\vskip .2in
The map $\mathfrak g\longrightarrow \mathfrak g$, given by $X\longmapsto -X$. 
This extends to a $\mathbb C$--linear
anti-automorphism $u\longmapsto u^{\#}$ of $\mathcal U(\mathfrak g_{\mathbb C})$
which satisfies   
\begin{equation}\label{srs-19}
\int_{\Gamma\backslash G} u.F(x)\ f(x)dx=\int_{\Gamma\backslash G} F(x)\ u^\#.f(x)dx
\end{equation}

\vskip .2in
 Since  $\mathcal S(\Gamma\backslash G)$ is a smooth representation, for each
 $u\in \mathcal U(\mathfrak g_{\mathbb C})$, the map
 $f\longmapsto u.f$ is continuous. So, if $\beta \in \mathcal S(\Gamma\backslash G)'$, then
 $f\longmapsto \beta(u.f)$ is a
 continuous linear functional.
 Hence, $\mathcal S(\Gamma\backslash G)'$ becomes $\mathcal U(\mathfrak g_{\mathbb C})$--module:
 $$
 u.\beta(f)=\beta(u^{\#}.f), \ \ f\in \mathcal S(\Gamma\backslash G).
 $$

 \vskip.2in
 
We consider the embedding of $L^2(\Gamma\backslash G)\hookrightarrow \mathcal S(\Gamma\backslash G)'$, given by
$\varphi\longmapsto \beta_\varphi$ where $\beta_\varphi$ is defined by $\beta_\varphi(f)=\int_{\Gamma \backslash G}
\varphi(x)f(x)dx$, for
$f\in \mathcal S(\Gamma\backslash G)$. It is proved in (\cite{casselman-1}, Proposition 1.17) that the map is
continuous. We sketch the
argument. Let $d>0$ be an integer  such that $\int_{\Gamma\backslash G} ||x||^{-2d}_{\Gamma\backslash G}dx<\infty$.
Let $B\subset  \mathcal S(\Gamma\backslash G)$ be a bounded set. Then, we have the following:

\begin{equation}\label{srs-17}
\begin{aligned}
  ||\beta_\varphi||_B&=\sup_{f\in B} \left|\int_{\Gamma \backslash G}\varphi(x)f(x)dx\right|\le
  \sup_{f\in B} \int_{\Gamma \backslash G}\left|\varphi(x)\right| \left|f(x)\right| dx\\
  &= \left(\int_{\Gamma\backslash G} ||x||^{-2d}_{\Gamma\backslash G}dx\right)^{1/2}
  \left(\sup_{f\in B} ||f||_{1, -d}\right) \cdot  \left(\int_{\Gamma \backslash G}\left|\varphi(x)\right|^2dx\right)^{1/2}
\end{aligned}
\end{equation}
which clearly proves the continuity. It is by the general theory that we have continuous map of smooth
representations
$\left(L^2(\Gamma\backslash G)\right)^\infty \hookrightarrow \left(\mathcal S(\Gamma\backslash G)'\right)^\infty$,
the image is actually
in $\mathcal A_{umg}(\Gamma\backslash G)$ (see Lemma \ref{src-5}). But even more is true

\begin{Lem}\label{srs-18}
  If the sequence   $(\varphi_n)_{n\ge 1}$ in $\left(L^2(\Gamma\backslash G)\right)^\infty$ converges to $\varphi\in
  L^2(\Gamma\backslash G)$, then for each  $u\in \mathcal U(\mathfrak g_{\mathbb C})$ the sequence
  $(\beta_{r(u)\varphi_n})_{n\ge 1}$ converges to $u.\beta_\varphi$ in the topology of $\mathcal S(\Gamma\backslash G)'$.
  \end{Lem}
\begin{proof}Arguing as in (\ref{srs-17}) and using (\ref{srs-19}), we have  
$$
 ||\beta_{u.\varphi_n}-u.\beta_\varphi||_B\le  \left(\int_{\Gamma\backslash G} ||x||^{-2d}_{\Gamma\backslash G}dx\right)^{1/2}
 \left(\sup_{f\in B} ||f||_{u^\#, -d}\right) \cdot  \left(\int_{\Gamma \backslash G}\left|\varphi_n(x)-\varphi(x)\right|^2dx\right)^{1/2},
 $$
 for all bounded sets $B\subset  \mathcal S(\Gamma\backslash G)$  and $u\in \mathcal U(\mathfrak g_{\mathbb C})$.
  \end{proof}

\vskip .2in
\begin{Cor}\label{srs-20}
 $\beta_\varphi$ is a  smooth vector in $\mathcal S(\Gamma\backslash G)'$ for all $\varphi\in L^2(\Gamma\backslash G)$.
  \end{Cor}
\begin{proof} We recall  that the space of smooth vectors $\left(\mathcal S(\Gamma\backslash G)'\right)^\infty$ in
  $\mathcal S(\Gamma\backslash G)'$
  is a complete locally convex representation of $G$ where topology is defined by the semi--norms:
$$
\alpha\longmapsto ||r'(u)\alpha||_B,
$$
where $u\in  \mathcal U(\mathfrak g_{\mathbb C})$ and $B\subset \mathcal S(\Gamma \backslash G)$ is bounded. Now,
since by the general theory,  the image of $\left(L^2(\Gamma\backslash G)\right)^\infty$
belongs to $\left(\mathcal S(\Gamma\backslash G)'\right)^\infty$, 
we may apply Lemma \ref{srs-18} to complete the proof.
\end{proof}

 \vskip .2in 

\begin{Lem}\label{srs-26}
Let $\mathcal H$ be a closed irreducible $G$--invariant subspace of $L^2(\Gamma\backslash G)$. Then, we have the
following  commutative diagram:
$$
\begin{CD}
  \mathcal H @>\subset>> L^2(\Gamma\backslash G) @>\varphi\longmapsto \beta_\varphi>>  \mathcal S(\Gamma\backslash G)'\\
  @AAA @AAA @AAA\\
  \mathcal H^\infty @>\subset >>  \left( L^2(\Gamma\backslash G) \right)^\infty @>\varphi\longmapsto \beta_\varphi>>
  \mathcal A_{umg}(\Gamma\backslash G), \\
\end{CD}
$$
where in the first row are continuous maps, and  the second row is also
  continuous if we equip $\mathcal A_{umg}(\Gamma\backslash G)$ with the topology inherited from $\left(\mathcal
   S(\Gamma\backslash G)'\right)^\infty$.
\end{Lem}
  \begin{proof} Above discussions imply that the first row consists of continuous maps. Next, 
    Lemma \ref{src-4} and Diximier--Malliavin theorem \cite{dixmal} assure that the image of
    $\left( L^2(\Gamma\backslash G) \right)^\infty$ in $\mathcal A_{umg}(\Gamma\backslash G)$. Finally, the commutativity of
    the diagram is a consequence of general facts about smooth vectors.
    \end{proof}
  \vskip .2in
  The following result uses deep results about globalization due to Casselman \cite{casselman} and Wallach
  \cite{W2}.

\vskip .2in
\begin{Lem}\label{srs-27}
  Let $\mathcal H$ be a closed irreducible $G$--invariant subspace of $L^2(\Gamma\backslash G)$. Then, there exists
  $n_0\in \mathbb Z$ such that for $n\ge n_0$, the map $\varphi\longmapsto \beta_\varphi$ maps  $\mathcal H^\infty$ equipped with its
  natural topology into $\mathcal A_{umg, n}(\Gamma\backslash G)$ (considered as a subspace of $\mathcal S(\Gamma\backslash G)'$ but equipped
  with its standard topology)
  isomorphically onto its image which is closed in
  $\mathcal A_{umg, n}(\Gamma\backslash G)$.
  \end{Lem}
\begin{proof} By Lemma \ref{srs-26}, the map $\mathcal H^\infty\longrightarrow \mathcal A_{umg}(\Gamma\backslash G)$, given by
  $\varphi\longmapsto \beta_\varphi$, is continuous if we equip $\mathcal A_{umg}(\Gamma\backslash G)$ with the topology
  inherited from $\left(\mathcal S(\Gamma\backslash G)'\right)^\infty$. It is also
  $\mathcal U(\mathfrak g_{\mathbb C})$--equivariant. Select any non--zero $\varphi\in \mathcal H_K$. Then, there exists
  $n_0\in \mathbb Z$ such that
  $$
  \varphi\in  \mathcal A_{umg, n_0}(\Gamma\backslash G)\subset \mathcal A_{umg, n_0+1}(\Gamma\backslash G)\subset
  \mathcal A_{umg, n_0+2}(\Gamma\backslash G)\subset \cdots.
  $$
 But since $\mathcal A_{umg, n}(\Gamma\backslash G)$ are smooth representations, and $\mathcal H_K$ is an irreducible
 $\mathcal U(\mathfrak g_{\mathbb C})$--module, we see that the image of $\mathcal H_K$ is contained in
 $\mathcal A_{umg, n}(\Gamma\backslash G)$ for $n\ge n_0$. Let $W_n$ be the closure of the image in
 $\mathcal A_{umg, n}(\Gamma\backslash G)$ for $n\ge n_0$. It is obvious that a $(\mathfrak g, K)$--module on the space of
 $K$--finite vectors  in $W_n$ is the image of $\mathcal H_K$. Therefore, $W_n$ is irreducible.
 We remark that $W_n$  being a closed subrepresentation of a
 representation of moderate growth $\mathcal A_{umg, n}(\Gamma\backslash G)$ is also a representation of moderate growth
 (see Lemma \ref{srs-24}; \cite{W2}, Lemma 11.5.2). But $\mathcal H^\infty$
 is also a representation of moderate growth (\cite{W2}, Lemma 11.5.1) and irreducible (see Lemma \ref{srs-02}).
 So,  the map $\mathcal H_K\longrightarrow \left(W_n\right)_K$,
 $\varphi\longrightarrow \beta_\varphi$, which is an isomorphism of $(\mathfrak g, K)$--modules,  extends to a continuous isomorphism
 of $G$--representations $\mathcal H^\infty$ and $W_n$ applying (\cite{W2}, Theorem 11.5.1). Let us finally determine this map. This is
 easy since by the composition
 with the continuous inclusion $\mathcal A_{umg, n}(\Gamma \backslash G)\hookrightarrow \left(\mathcal
   S(\Gamma\backslash G)'\right)^\infty$ (see Lemma \ref{srs-28}), we obtained the map that coincides on $\mathcal H_K$ 
   with the continuous map given by the second row of the diagram in Lemma \ref{srs-26}. Hence, the map is $\varphi
   \longrightarrow \beta_\varphi$.
  \end{proof}

\vskip .2in
The first main result of this section is the following theorem. The reader should review the
statement of Proposition \ref{srs-0}.

\vskip .2in

  \begin{Thm}\label{srs-21}
    Let $\mathcal H$ be a closed irreducible $G$--invariant subspace of $L^2(\Gamma\backslash G)$. Then, we have the following:

    \begin{itemize}
      \item[(i)] The continuous inclusion
  $\mathcal H\hookrightarrow L^2(\Gamma\backslash G)$ gives rise to a continuous linear functional $\lambda$
  such that the following diagram is commutative:
$$
\begin{CD}
  \mathcal H  @>\varphi\longmapsto \beta_\varphi>>  \mathcal S(\Gamma\backslash G)'\\
  @A\subset AA @A=AA\\
  \mathcal H^\infty @>\alpha_{\lambda, \Gamma} >>    \mathcal S(\Gamma\backslash G)'. \\
\end{CD}
$$
Furthermore,   $\mathcal H$ is embedded into  the smooth vectors of the closure
$Cl\left(\alpha_{\lambda, \Gamma}\left(\mathcal H^\infty\right) \right)$.
\item[(ii)] In addition, assume that Assumption \ref{paf-00} holds. Then, if $\mathcal H$ is tempered, then  $d_{\mathcal H, \lambda}=-\infty$. 
  \end{itemize}
\end{Thm} 
  \begin{proof} We prove (i). Lemma \ref{srs-26} implies that $\mathcal H^\infty \subset \mathcal A_{umg}(\Gamma \backslash G)$. Next, by Lemma
    \ref{srs-27}, there exists an integer $n\ge 1$ such that $\mathcal H^\infty \subset \mathcal A_{umg, n}(\Gamma \backslash G)$.
    This inclusion is continuous in appropriate topologies. Hence, by Lemma \ref{srs-22}, $\varphi\longmapsto \varphi(1)$ is
    a continuous functional on  $\mathcal H^\infty$. If we denote this functional by $\lambda$, then the commutativity of the diagram
    follows. The last part of (i) follows from Corollary \ref{srs-20}.

    We prove (ii). Because of the Assumption \ref{paf-00}, we may consider the space of the closed subspace
    $L^2_{cusp}(\Gamma\backslash G)$ of cuspidal functions in $L^2(\Gamma\backslash G)$. It is a $G$--subrepresentation. By a result of
    Wallach \cite{W0}, since $\mathcal H$ is a tempered closed subrepresentation of  $L^2(\Gamma\backslash G)$,  $\mathcal H$ is a closed
    subrepresentation of  $L^2_{cusp}(\Gamma\backslash G)$. Then, using the notation of Section \ref{paf},
    $\mathcal H_K\subset \mathcal A_{cusp}(\Gamma\backslash G)$, and in fact
    $$
    \mathcal H_K\subset \mathcal A_{cusp}(\Gamma\backslash G)\cap \mathcal S(\Gamma\backslash G),
    $$
    since $K$--finite cuspidal automorphic forms are rapidly decreasing \cite{BJ}.  Now, arguing as in the proof of Lemma \ref{srs-27}, we see
    that
    $$
    \mathcal H^\infty \subset \mathcal A^\infty_{cusp}(\Gamma\backslash G)\cap \mathcal S(\Gamma\backslash G).
    $$
    This implies (ii). It also shows that smooth cuspidal automorphic forms are rapidly decreasing. Which gives a different proof of
    the fact proved also in \cite{ms1}.
  \end{proof}

\vskip .2in
We maintain the Assumption \ref{paf-00}, and assume that $G$ posses representations in discrete series
(\cite{milicic}, \cite{hc}). Then, if $(\pi, \mathcal H)$ is a representation in discrete series, then there exists infinitely many
congruence subgroups $\Gamma$ such that we can embedded it in $L^2_{cusp}(\Gamma\backslash G)$ (\cite{savin}, \cite{clozel}). Therefore, it
posses a non--zero $\Gamma$ invariant functional such that $d_{\pi, \lambda}=-\infty$. On the other hand, by counting tempered representation,
most of them do not appear as subrepresentations of $L^2(\Gamma\backslash G)$ for a congruence subgroup $\Gamma$.

\vskip .2in
Following Harish--Chandra (\cite{hc}, Section 5), we introduce the topology on $C^\infty(G)$ by means of seminorms
$$
\nu_{\Omega, u}, \ \ \text{$\Omega\subset G$ is compact and $u\in \mathcal U(\mathfrak g_{\mathbb C})$}
$$
defined by
$$
  \nu_{\Omega, u}(F)=\sup_{x\in \Omega}{\ ||uF(x)||}.
$$

\vskip .2in

We end this section with the following theorem:

\begin{Thm}\label{srs-32}
  Let $\mathcal H$ be a closed irreducible $G$--invariant subspace of $L^2(\Gamma\backslash G)$. Assume that the sequence of elements
  in $\mathcal H^\infty$, $(\varphi_n)_{n\ge 1}$, converges to $\varphi\in\mathcal H^\infty$ in the standard topology on $\mathcal H^\infty$.
  Then, it converges to $\varphi$ in  above described topology on $C^\infty(G)$. In addition, assume that Assumption \ref{paf-00} holds.
  Then, if $\mathcal H\subset L^2_{cusp}(\Gamma\backslash G)$, then $u.\varphi_n\longmapsto u.\varphi$ uniformly on $G$ for all
  $u\in \mathcal U(\mathfrak g_{\mathbb C})$.
\end{Thm}
\begin{proof} The first part follows from Lemma \ref{srs-27}. In addition, for the second part, we need the fact that
  $$
    \mathcal H^\infty \subset \mathcal A^\infty_{cusp}(\Gamma\backslash G)\cap \mathcal S(\Gamma\backslash G)
    $$
    established in the proof of Theorem \ref{srs-31}.
  \end{proof}
    
\section{On $\Gamma$--invariants in $\mathcal S'(G)$}\label{oii}
\medskip

\medskip

Let $\Gamma \subset G$ be a discrete subgroup. Then, the canonical map $\mathcal S(G)\longrightarrow \mathcal S(\Gamma \backslash G)$, given by 
$P_\Gamma(f)(x)=\sum_{\gamma\in \Gamma} f(\gamma x)$, is a continuous  (\cite{casselman-1}, Proposition 1.110).  We sketch the argument since the
details of the argument will be useful later. Let  $u\in \mathcal U(\mathfrak g_{\mathbb C})$. Let $n\ge 1$ be an integer. Then, we have
$$
\left|\left|P_\Gamma(f)\right|\right|_{u, -n}= \sup_{x\in G}  ||x||_{\Gamma\setminus G}^{n} \left|u.P_\Gamma(f)(x)\right|.
$$
Since $u.P_\Gamma(f)=P_\Gamma(u.f)$ and $||x||_{\Gamma\setminus G}\le ||\gamma x||$, we obtain
\begin{align*}
||x||_{\Gamma\setminus G}^{n} \left|u.P_\Gamma(f)(x)\right|&=||x||_{\Gamma\setminus G}^{n} \left|P_\Gamma(u.f)(x)\right|\\
&\le 
\sum_{\gamma\in \Gamma} ||\gamma x||^n \cdot |u.f(\gamma x)|\le \left|\left|f\right|\right|_{u, -d-n}
\left(\sum_{\gamma\in \Gamma} ||\gamma x||^{-d}\right),
\end{align*}
where $d>0$ is large enough such that $\int_G ||x||^{-d} dx<\infty$. But, by (\cite{casselman-1}, Lemma 1.10), we have
   \begin{equation}\label{srs-13}
M_d\overset{def}{=}\sup_{x\in G}  \sum_{\gamma \in \Gamma} \left|\left|\gamma x\right|\right|^{-d} <\infty.
\end{equation}
 Thus, we obtain
 \begin{equation}\label{srs-14}
\left|\left|P_\Gamma(f)\right|\right|_{u, -n}\le M_d \left|\left| f\right|\right|_{u, -d-n}.
\end{equation}

\medskip
The group $\Gamma$ acts on the left on $\mathcal S(G)$:  $l(\gamma)f(x)=f(\gamma^{-1}x)$.
By duality $\Gamma$ acts on $\mathcal S(G)'$: $l'(\gamma)\alpha(f)=\alpha(l(\gamma^{-1})f)$.

\medskip
\begin{Lem}\label{srs-15}
  Let $\gamma \in \Gamma$. Then, 
  the linear operator $l(\gamma)$ (resp., $l'(\gamma)$)  is continuous in the topology on $\mathcal S(G)$
  (resp.,  $\mathcal S(G)'$).
\end{Lem}
\begin{proof}
Indeed,  
for $u\in \mathcal U(\mathfrak g_{\mathbb C})$, and for  an integer $n\ge 1$,  we have the following:
$$
||l(\gamma)f||_{u, -n}=\sup_{x\in G} ||x||^n |u.f(\gamma^{-1}x)|=\sup_{x\in G} ||\gamma x||^n |u.f(x)|\le
||\gamma||^n ||f||_{u, -n}.
$$
This proves that $l(\gamma)$ is continuous. Next, we have 
\begin{align*}
||l'(\gamma)\alpha||_B &=\sup_{f\in B} \left|\alpha(l(\gamma^{-1})f)\right|=\sup_{f\in l(\gamma^{-1})B} \left|\alpha(f)\right|=||
\alpha||_{l(\gamma^{-1})B}.
\end{align*}
We remark that since $l(\gamma^{-1})$ is continuous, the set $l(\gamma^{-1})B$ is bounded.
This proves that $l'(\gamma)$ is also continuous.
\end{proof}

\medskip
\begin{Lem}\label{srs-16}
  Let  $\left(\mathcal S(G)'\right)^\Gamma$ be the space of all $\alpha\in \mathcal S(G)'$ such that $l'(\gamma)\alpha=\alpha$ for all
  $\gamma\in \Gamma$. Then,  $\left(\mathcal S(G)'\right)^\Gamma$ is a closed subrepresentation of $\mathcal S(G)$
  (where $G$ acts by right translations).
\end{Lem}
\begin{proof} Indeed, if $(\alpha_\lambda)_{\lambda\in \Lambda}$
is a net in  
$\left(\mathcal S(G)'\right)^\Gamma$ which converges to $\alpha  \in \mathcal S(G)'$ i.e., the nets $||\alpha_\lambda-\alpha||_B$,
where $B\subset \mathcal S(G)$ is bounded, 
converge to zero.  Then, since for  $\gamma\in \Gamma$, the operator $l'(\gamma)$ is continuous,
we have that the net $l'(\gamma)\alpha_\lambda$ converges to $l'(\gamma)\alpha$.
This implies $l'(\gamma)\alpha=\alpha$. Hence, $\alpha\in \left(\mathcal S(G)'\right)^\Gamma$.
\end{proof}
\medskip

\begin{Prop}\label{srs-12} We maintain  Assumption \ref{paf-00}.  Then, the canonical map
  $\mathcal S(\Gamma \backslash G)'\longrightarrow \mathcal S(G)'$
  is a continuous embedding  with the image dense in the closed subrepresentation $\left(\mathcal S(G)'\right)^\Gamma$.
  The space $\mathcal A_{umg}(\Gamma \backslash G)$ gets 
identified with the Garding space of  the subrepresentation $\left(\mathcal S(G)'\right)^\Gamma$.
\end{Prop}
\begin{proof}  Since  Assumption \ref{paf-00} holds,
  the canonical map $\mathcal S(G)\longrightarrow \mathcal S(\Gamma \backslash G)$, given by 
$P_\Gamma(f)(x)=\sum_{\gamma\in \Gamma} f(\gamma x)$, is a continuous epimorphism (\cite{casselman-1}, Proposition 1.11, Theorem 2.2).  

Next, the map $\mathcal S(\Gamma \backslash G)'\longrightarrow \mathcal S(G)'$ is an embedding. It is also obvious that 
its image is contained in  $\left(\mathcal S(G)'\right)^\Gamma$.  Let us that it is continuous. Let $B\subset \mathcal S(G)$
be a bounded set.  Then, since $P_\Gamma$ is continuous, $P_\Gamma\left(B\right)\subset \mathcal S(\Gamma \backslash G)$ is a bounded set. 
Then, we have 
$$
||\alpha\circ P_\Gamma||_B=\sup_{f\in B} \left| \alpha\left(P_\Gamma(f)\right)\right| =||\alpha||_{P_\Gamma\left(B\right)}.
$$
This proves the continuity of the map.

The space $\mathcal A_{umg}(\Gamma \backslash G)$ is the Garding space of $\mathcal S(\Gamma \backslash G)'$.
Thus its image is contained in the Garding space 
of  the subrepresentation $\left(\mathcal S(G)'\right)^\Gamma$. But the Garding space of  $\left(\mathcal S(G)'\right)^\Gamma$ is
contained in the Garding space of
$\mathcal S(G)'$. This space is $\mathcal A_{umg}(G)$ (see Lemma \ref{src-4}). So let $\alpha$ belong to the Garding space of
$\left(\mathcal S(G)'\right)^\Gamma$.
Then, by what we have just said, 
$\alpha$ is represented by a function $\varphi\in  \mathcal A_{umg}(G)$:
$$
\alpha(f)=\int_G \varphi(x) f(x)dx, \ \ f\in \mathcal S(G).
$$
Since $\alpha$ is $\Gamma$--invariant, we have that $\varphi(\gamma x)=\varphi(x)$, $\gamma\in \Gamma$, $x\in G$. Now, 
$\varphi\in  \mathcal A_{umg}(\Gamma \backslash G)$. 

Finally, since  $\mathcal A_{umg}(\Gamma \backslash G)$  maps onto the Garding space of $\left(\mathcal S(G)'\right)^\Gamma$, 
the space 
$\mathcal S(\Gamma \backslash G)'\longrightarrow \mathcal S(G)'$ maps onto a dense sbspace of $\left(\mathcal S(G)'\right)^\Gamma$.
\end{proof}

\medskip
In the following proposition we give the most general construction of classical Poincar\' e series. In part, it generalizes
(\cite{MuicSMOOTH}, Theorem 6.4).

\begin{Prop}\label{srs-11} Assume that $\Gamma \subset G$ is a discrete subgroup. Let $\varphi\in L^1(G)$. Then
  the series $\sum_{\gamma \in \Gamma} l(\gamma)\varphi$ converges absolutely in $\mathcal S(G)'$ to an element of
  $\left(\mathcal S(G)'\right)^\Gamma$ (which is in the image of  $\mathcal S(\Gamma \backslash G)'$).
  Moreover, if $\varphi$ is a smooth vector in the Banach representation  $L^1(G)$ under right--translations, then
  $\sum_{\gamma \in \Gamma} l(\gamma)\varphi\in \mathcal A_{umg}(\Gamma \backslash G)$.
  \end{Prop}
\begin{proof} Let $B\subset \mathcal S(G)$ be a bounded set. We need to show that
  $$\sum_{\gamma \in \Gamma} ||l(\gamma)\varphi||_B<\infty.
  $$
  Since $\mathcal S(\Gamma \backslash G)'$ is complete, this proves the absolute convergence.
  
  By definition, we have
  \begin{align*}
  ||l(\gamma)\varphi||_B=\sup_{f\in B} \ \ \left|\int_G \varphi(\gamma^{-1} x) f(x)dx\right|
  & \le \sup_{f\in B} \int_G \left|\varphi(\gamma^{-1} x)\right| \left|f(x)\right| dx\\
  &=\sup_{f\in B}
  \int_G \left|\varphi(x)\right| \left|f(\gamma x)\right| dx\\
  &\le \left(\sup_{f\in B} ||f||_{1, -d}\right)\cdot \int_G \left|\varphi(x)\right| \left|\left|\gamma x\right|\right|^{-d} dx
  \end{align*}
  So,  the series is
  $$
  \le \left(\sup_{f\in B} ||f||_{1, -d}\right)\cdot 
  M_d  \int_G \left|\varphi(x)\right| dx <\infty,
  $$
  where the number $M_d$ is defined by (\ref{srs-13}).

The distribution in question is in fact the integration against the classical Poincar\' e series
$P_\Gamma(\varphi)\in L^1(\Gamma\backslash G)$:
\begin{align*}
\int_G P_\Gamma(\varphi)(x)f(x)dx&=\sum_{\gamma\in \Gamma} \int_G \varphi(\gamma x) f(x) dx\\
&=
\sum_{\gamma\in \Gamma} \int_G \varphi(\gamma^{-1} x) f(x) dx= \sum_{\gamma\in \Gamma} 
\int_G l(\gamma)\varphi(x) f(x) dx, 
\end{align*}
for  $f\in \mathcal S(G)$.

The space of smooth vectors in $L^1(G)$, where $G$ acts by right translations $r$, is a Frech\' et space under seminorms
(\cite{W2}, Lemma 11.5.1):
$$
||r(u)f||_1=\int_G \left|r(u)f(x)\right| dx, \ \ u\in \mathcal U(\mathfrak g_{\mathbb C}).
$$
Then, by Diximier--Malliavin theorem \cite{dixmal}, for smooth vector $\varphi$ there exists, smooth vectors $\varphi, \ldots, \varphi_l$, and
$\alpha_1, \ldots, \alpha_l\in C_c^\infty (G)$ such that
$$
\varphi=\sum_{i=1}^l r(\alpha_i)\varphi_i=\sum_{i=1}^l \varphi_i\star \alpha^\vee_i
$$
where as usual $\alpha^\vee_i(x)=\alpha_i(x^{-1})$. By the standard measure--theoretic arguments, we have
$$
P_\Gamma(\varphi)=\sum_{i=1}^l P_\Gamma(\varphi_i)\star \alpha^\vee_i.
$$
Now, we apply Lemma \ref{src-5}.
\end{proof}

\end{document}